\documentclass[11pt]{amsart}
\usepackage{amsmath,amsfonts,amssymb,amsthm,amscd,latexsym,euscript,verbatim}
\usepackage[all]{xy}

\makeatletter
\def\section{\@startsection{section}{1}%
  \z@{1.1\linespacing\@plus\linespacing}{.8\linespacing}%
  {\normalfont\Large\scshape\centering}}
\makeatother

\theoremstyle{plain}

\newtheorem*{MT}{Main Theorem}

\newtheorem*{conj*}{Root Groups Conjecture}
\newtheorem*{thm1.2}{(1.2) Theorem}
\newtheorem*{thm1.3}{(1.3) Theorem}
\newtheorem*{thm1.4}{(1.4) Theorem}
\newtheorem*{prop*}{Proposition}

\newtheorem{prop}{Proposition}[section]

\newtheorem{lemma}[prop]{Lemma}

\theoremstyle{definition}

\newtheorem*{Def*}{Definition}
\newtheorem{Defs}[prop]{Definitions}

\newtheorem*{notation*}{Notation}
\newtheorem{remark}[prop]{Remark}

\newtheorem*{remark*}{Remark}

\newcommand{\one}{{\bf 1}}

\numberwithin{equation}{section}

\begin{document}
\title[]{Alternative rings whose associators are not zero-divisors}
\author[Kleinfeld and Segev]{Erwin Kleinfeld\qquad Yoav Segev}
\address{Yoav Segev \\
         Department of Mathematics \\
        Ben-Gurion University \\
        Beer-Sheva 84105 \\
         Israel}
\email{yoavs@math.bgu.ac.il}

\address{Erwin Kleinfeld \\
1555 N.~Sierra St.~Apt 120, Reno, NV 89503-1719, USA}
\email{erwinkleinfeld@gmail.com}

\begin{abstract}
The purpose of this short note is to prove that if
$R$ is an alternative ring whose associators are not zero-divisors,
then $R$ has no zero divisors.  By a result of Bruck and Kleinfeld,
if, in addition, the characteristic of $R$ is not $2,$ then 
the central quotient of $R$ is an octonion division algebra over some field.
\end{abstract}

\date{\today}
\maketitle

%
\section{Introduction}

In this paper we charaterize  proper (i.e.~not associative) alternative rings 
whose central quotient (in the sense of \cite{BK}) is an octonion division algebra over some field,
as in the Main Theorem.
This adds to another recent characterization that we proved in \cite{KS}. 
For historical information on the octonions see \cite{Bl, E},
and for their connection with algebraic groups see \cite{SV}.

Before we state our main result, we draw the
attention of the reader to Remark \ref{rem 3.2} with
regards to the notion of a zero divisor.

\begin{MT}\label{MT}
Let $R$ be an alternative, not associative ring.
Assume that
a non-zero  associator in $R$  is not a divisor of zero. Then
\begin{itemize}
\item[(a)]
$R$ contains no divisors of zero.

\item[(b)]
If, in addition, the characteristic of $R$ is not $2,$
then the central quotient ring $R//C$
is an $8$-dimensional octonion division algebra
over its center--the fraction field of the
center $C$ of $R.$
\end{itemize}
\end{MT}

%
\section{Preliminaries on alternative rings}

Our main reference for alternative rings is \cite{K}.
Let $R$ be a ring, not necessarily with $\one$ and not
necessarily associative.  

\begin{Defs}
Let $x,y,z\in R.$
\begin{enumerate}
\item
The {\bf associator} $(x,y,z)$ is defined to be
\[
(x,y,z)=(xy)z-x(yz).
\]

\item
The {\bf commutator} $(x,y)$ is defined to be
\[
(x,y)=xy-yx.
\]

\item
$R$ is an {\bf alternative ring} if
\[
(x,y,y)=0=(y,y,x),
\]
for all $x,y\in R.$

\item
The {\bf nucleus} of $R$ is denoted $N$ and defined
\[
N=\{n\in R\mid (n,R,R)=0\}.
\]
Note that in an alternative ring the associator is skew symmetric in its $3$
variables (\cite[Lemma 1]{K}).  Hence   $(R,n,R)=(R,R,n)=0,$ for $n\in N.$

\item
The {\bf center} of $R$ is denoted $C$ and defined
\[
C=\{c\in N\mid (c,R)=0\}.
\]
\end{enumerate}
\end{Defs}

In the remainder of this paper {\bf $R$ is an alternative ring}
which is {\bf not associative}. 
$N$ denotes the nucleus of $R$ and $C$ its center.

The following lemma gives some properties of $R$ which we will require later.

\begin{lemma}\label{p of R}
Let $w,x,y,z\in R,$ and $n,n'\in N,$ then
\begin{enumerate}
\item
$(x,y,z)x=(x,xy,z)=(x,y,xz).$

\item
$x(x,y,z)=(x,yx,z)=(x,y,zx)$

\item
$(xn,y,z)=(nx,y,z)=(x,y,z)n=n(x,y,z).$

\item
$[N,R]\subseteq N.$

\item
Nuclear elements commute with associators.

\item
$(w,n)(w,n)(x,y,z)=0.$ 
\end{enumerate}
\end{lemma}
\begin{proof}
(1\&2)\quad 
This is in \cite[equation (4), p.~130]{K}, and \cite[equation (5), p.~130]{K}, respectively.
\medskip

\noindent
(3)\quad
By \cite[equation (2), p.~129]{K}, $(nx,y,z)=n(x,y,z),$
and by the beginning of the proof of Lemma 4, p.~132 in \cite{K},  
$(nx,y,z)  =  (x,y,z)n$ and $(xn,y,z)  =  n(x,y,z).$
\medskip

\noindent
(4\&5)\quad
By (3), $(nx,y,z)=(xn,y,z),$ hence (4) holds, and also $n(x,y,z)=(x,y,z)n,$ so (5) holds.
\medskip

\noindent
(6)\quad
This is \cite[Lemma 2.4(5)]{KS}.
\end{proof}

\section{Proof of the Main Theorem.}
In this section $R$ is an alternative ring which is not associative.
We assume that 
\[
\text{non-zero associators are not zero-divisors in R.}
\]

\begin{lemma}\label{p of tr}
\begin{enumerate}
\item
$N$ is contained in the center of $R.$
So $N=C.$

\item
If $n\in N$ is non-zero then $n$ is not a zero divisor.

\item
If $x\in R$ and $x\notin N,$ then $x$ is not a zero divisor.
\end{enumerate}
\end{lemma}
\begin{proof}
(1)\quad
This follows from Lemma \ref{p of R} parts (3), (4) and (6).
Indeed let $(x,y,z)\in R$ be a non-zero associator,
and let $(w,n)$ be a commutator, with $w\in R,$ and $n\in N.$
If $(w,n)\ne 0,$ 
then $(w,n)(x,y,z)\ne 0,$ but since $(w,n)\in N,$ we have  $(w,n)(x,y,z)=((w,n)x,y,z)\ne 0,$
and $(w,n)((w,n)x,y,z)=0,$ a contradiction.
\medskip

\noindent
(2)\quad
Let $n\in N$ be non-zero.  Suppose $tn=nt=0,$ for some $t\in R.$
Let $(x,y,z)$ be a non-zero associator.  Then $(x,y,z)n\ne 0,$ and using Lemma \ref{p of R}(3) we get 
\[
0=(x,y,z)(nt)=\big((x,y,z)n\big)t=(xn,y,z)t,
\]
so $t=0.$
\medskip

\noindent
(3)\quad
Let $x\in R$ such that $x\notin N.$  Assume that $xt=0,$ for some $t\in R.$
Since $x\notin N,$ there are $y,z\in R$ such that $u:=(x,y,z)\ne 0.$
Consider $w:=(u,x,t)=(ux)t$ (because $xt=0).$  
By Lemma \ref{p of R}(1) $ux=(x,xy,z)\ne 0.$
If $(u,x,t)= 0,$ we would get $(x,xy,z)t=0,$ so $t=0.$

Suppose $(u,x,t)\ne 0,$ so that $t\ne 0.$ Then, by Lemma \ref{p of R}(1),
$(u,x,t)x=(u,x,xt)=0,$ this is a contradiction. 

Similarly, if we assume that $tx=0,$ then arguing as above
replacing $w$ with $w'=(t,x,u),$ and using Lemma \ref{p of R}(2),
we get that $t=0.$
\end{proof}

\begin{proof}[Proof of the Main Theorem]\hfill
\medskip

\noindent
By Lemma \ref{p of tr}(2\&3), $R$ has no zero divisors, this shows part (a).
Part (b) follows from
\cite[Theorem A]{BK}.
\end{proof}

\begin{remark}\label{rem 3.2}
As usual, a non-zero element $x\in R$ is a {\it left zero divisor} if there
exists a non-zero $y\in R$ 
such that $xy=0.$ {\it Right zero divisors} are similarly defined.
In this paper A non-zero element in $R$ is a {\it zero divisor} if it is
either a left zero divisor or a right zero divisor.
  
Thus, when we say that an element is 
{\it not a zero divisor} we mean that it is both not a left zero
divisor and not a right zero divisor.

Note that our proof shows that if, in the Main Theorem, we would replace
``zero divisor'' by ``left zero divisor'', then part (a) of the theorem 
would still hold.  Similarly, part (a) would hold if we would replace
``zero divisor'' by ``right zero divisor''.
\end{remark}



\end{document}